\documentclass[11pt]{amsart}

\usepackage{latexsym, amssymb,enumerate}
\usepackage[T1]{fontenc}
\usepackage{textcomp}
\usepackage{appendix}

\newcommand{\be}{\begin{equation}}
\newcommand{\ee}{\end{equation}}
\newcommand{\beq}{\begin{eqnarray}}
\newcommand{\eeq}{\end{eqnarray}}

\def\H{{\Bbb H}}

\def\R{{\mathfrak R}}

\def\p{\partial}
\def\S{\Sigma}

\def\<{\langle}
\def\>{\rangle}
\newtheorem{prop}{Proposition}[section]
\newtheorem{theo}[prop]{Theorem}
\newtheorem{lemm}[prop]{Lemma}
\newtheorem{coro}[prop]{Corollary}

\newtheorem{RK}{Remark}
\def\begeq{\begin{equation}}
\def\endeq{\end{equation}}
\def\p{\partial}

\def\R{\Bbb R}
\def\RR{\mathbb R}\def\R{\mathbb R}

\def\tr{{\rm tr}}

\def\d{\delta}

\def\s{\sigma}
\def\l{\lambda}
\def\L{\Lambda}
\def \ds{\displaystyle}

\def\odot{\setbox0=\hbox{$\bigcirc$}\relax \mathbin {\hbox
to0pt{\raise.5pt\hbox to\wd0{\hfil $\wedge$\hfil}\hss}\box0 }}

\numberwithin{equation} {section}

\begin{document}
\author{Jie Wu}
\address{School of Mathematical Sciences, University of Science and Technology
of China Hefei 230026, P. R. China
\and
 Albert-Ludwigs-Universit\"at Freiburg,
Mathematisches Institut
Eckerstr. 1,
D-79104, Freiburg, Germany
}
\email{jie.wu@math.uni-freiburg.de}
\thanks{The first author is partly supported by SFB/TR71
``Geometric partial differential equations''  of DFG}
\author{Chao Xia}\address{Max-Planck-Institut f\"ur Mathematik in den Naturwissenschaften, Inselstr. 22, D-04103, Leipzig, Germany}\thanks{The second author is supported by funding from the European Research Council
under the European Union's Seventh Framework Programme (FP7/2007-2013) / ERC grant
agreement no. 267087.}

\email{chao.xia@mis.mpg.de}
\begin{abstract}{
In this paper, we first investigate several rigidity problems for hypersurfaces in the warped product manifolds with constant linear combinations of higher order mean curvatures as well as ``weighted'' mean curvatures, which extend the work \cite{Mon, Brendle,BE} considering constant mean curvature functions. Secondly,
we obtain the rigidity results for hypersurfaces in the space forms with constant linear combinations of  intrinsic Gauss-Bonnet curvatures $L_k$. To achieve this, we develop some new kind of Newton-Maclaurin type inequalities on $L_k$ which may have independent interest.
}

\end{abstract}

\title[On Rigidity of hypersurfaces with constant curvature functions]{On Rigidity of hypersurfaces with constant curvature functions in warped product manifolds}
\maketitle

\section{Introduction}

The rigidity problem of hypersurfaces with constant curvature functions has attracted much attention in the classical differential geometry. The most typical curvature functions are the \textit{extrinsic} mean curvature and the \textit{intrinsic} Gauss (scalar) curvature. In 1899, Liebmann \cite{Li} showed two rigidity results that closed surfaces with constant Gauss curvature or \textit{convex} closed surfaces with constant mean curvature in $\mathbb{R}^3$ are spheres. Later, S\"uss \cite{Su} and Hsiung \cite{Hs} proved the rigidity for \textit{convex} or \textit{star-shaped} hypersurfaces in $\mathbb{R}^n$ for all $n$. In later 1950s, the condition of convexity or star-shapedness was eventually removed by Alexandrov in a series of papers \cite{A}. Namely, he proved that closed hypersurfaces with constant mean curvature \textit{embedded} in the Euclidean space are spheres. This result is now often referred to as \textit{Alexandrov Theorem}. Also his method, based on the maximum principle for elliptic equations, is totally different with all previous ones and now referred to as  \textit{Alexandrov's reflection method}. The embeddedness condition is necessary in view of the famous counterexamples provided by Hsiang-Teng-Yu \cite{HTY} and Wente \cite{Wente}. After the work of Alexandrov, lots of extensions appeared on such rigidity topic. Montiel and Ros \cite{Ros0, Ros, MR} proved results for hypersurfaces with constant higher order mean curvatures \textit{embedded} in space forms, following the work of Reilly \cite{Reilly} who recovered Alexandrov Theorem by using an integral technique. Simultaneously, Korevaar \cite{Kor} proved the same results following the method of Alexandrov.  Later, Montiel \cite{Mon} studied the same problem in more general ambient manifolds, the warped product manifolds. His result was in fact Hsiung's type since he added the condition of star-shapedness to the corresponding hypersurfaces.
Quite recently, Brendle \cite{Brendle} removed this star-shapedness condition and hence proved Alexandrov Theorem for constant mean curvature hypersurfaces in general warped product manifolds, including the (Anti-)deSitter-Schwarzschild manifolds as a typical example. Thereafter,  Brendle and Eichmair \cite{BE} extended the result to any compact star-shaped hypersurfaces with constant higher order mean curvature, where star-shapedness is needed again.  For other generalizations, see  for instance \cite{AIR,AL, ALM, HLMG, HMZ, Mon2} and references therein.

In this paper, we first investigate several related rigidity problems for hypersurfaces with constant curvature functions \textit{embedded} in the warped product manifolds.

Let us start with the setting. Assume $(N^{n-1}(K), g_N)$ is an $(n-1)$-dimensional compact manifold with constant sectional curvature $K$.
Let $(M^n,\bar{g})$ be an $n$-dimensional ($n\geq 3$) warped product manifold $M= [0,\bar{r})\times_\l N(K)$ $(0<\bar{r}\leq \infty)$, equipped with a Riemannian metric $$\bar{g}=dr^2+\lambda(r)^2 g_N,$$ where $\l: [0,\bar{r})\to \RR$ is a smooth positive function satisfying the following conditions:

\begin{itemize}
\item[(C1)] $\lambda'(r)>0$ for all $r\in (0,\bar{r})$;
\item[(C2)] $\frac{\lambda''(r)}{\lambda(r)}+\frac{K-\lambda'(r)^2}{\lambda(r)^2}>0$ for all $r\in (0,\bar{r})$;
\item[(C3)] $\lambda''(r)\geq0$ for all $r\in (0,\bar{r})$;
\item[(C4)] $\lambda'(0)=0$, $\lambda''(0)>0$; $2\frac{\lambda''(r)}{\lambda(r)}-(n-2)\frac{K-\lambda'(r)^2}{\lambda(r)^2}$ is non-decreasing for $r\in (0,\bar{r})$.
\end{itemize}
Condition (C2) is equivalent that Ricci curvature is smallest in the radial direction and the latter part of  (C4) is equivalent that scalar curvature is non-decreasing with respect to $r$ (see (\ref{Ric}) below). As shown in \cite{Brendle}, the Schwarzschild, the (Anti-)deSitter-Schwarzschild and the Reissner-Nordstrom manifolds satisfy (C1)-(C4).

Before stating our results, let us give some notations and terminologies.  For a hypersurface $\S$ in $M$, we denote by $H_k=H_k(\l)$ the normalized $k$-th mean curvature of $\S$, i.e.,
\begin{eqnarray}
H_k(\l)=\frac{1}{\binom{n-1}{k}}\sigma_k(\l),
\end{eqnarray}
where $\l=(\l_1,\cdots,\l_{n-1})$ are the principal curvatures of $\S$ and $\sigma_k$ is the $k$-th elementary symmetric function. We say that $\S$ is $k$-convex if  $\l$ satisfies $\sigma_j(\l)\geq 0$ for any $1\leq j\leq k$.  $\S$ is called star-shaped if $\<\frac{\p}{\p r},\nu\>\geq0$ for the outward normal $\nu$ of $\S$.

Our first result is on hypersurfaces with constant curvature quotients in the warped product manifolds. This kind of rigidity in the space forms can be obtained by using Alexandrov reflection method, which
was already referred by Korevaar \cite{Kor}.  Koh \cite{Koh2} gave another proof based on the Minkowski integral formula.

\begin{theo}\label{main thm} Let $(M^n,\bar{g})$ be an $n$-dimensional ($n\geq 3$) warped product manifold  satisfying (C1) and (C2).
Let $1\leq l<k\leq n-1$ be two integers and  $\S$ be a closed, star-shaped hypersurface in $(M, \bar{g})$. If there exists some constant $c$ such that  $H_l$ is nowhere vanishing and $\frac{H_k}{H_l}\equiv c$, then $\S$ is a slice $N\times \{r\}$ for some $r\in (0,\bar{r})$.
\end{theo}

Next, we study the rigidity problem for hypersurfaces with constant linear combinations of mean curvatures in the warped product manifolds.
\begin{theo}\label{mainthm0}
Let $(M^n,\bar{g})$ be an $n$-dimensional ($n\geq 3$) warped product manifold satisfying  (C1) and (C2). Let $0\leq l<k\leq n-1$ be two integers and $\Sigma$ be a closed, k-convex star-shaped hypersurface in $(M^n,\bar g)$.
If either of the following holds:
\begin{itemize}
\item[(i)] $2\leq l<k\leq n-1$ and  there are nonnegative constants $\{a_i\}_{i=1}^{l-1}$ and $\{b_j\}_{j=l}^k$, at least one of them not vanishing,  such that $$\sum_{i=1}^{l-1}a_iH_i=\sum_{j=l}^k b_j H_j;$$
\item[(ii)] $1\leq l<k\leq n-2$ and there are nonnegative constants $\{a_i\}_{i=0}^{l-1}$ and $\{b_j\}_{j=l}^k$,  at least one of them not vanishing, such that $$\sum_{i=0}^{l-1}a_iH_i=\sum_{j=l}^k b_j H_j;$$
\end{itemize}
then $\Sigma$ is a slice $N\times \{r\}$ for some $r\in (0,\bar{r})$.
\end{theo}
%We will use the integral technique following \cite{Reilly} to prove Theorem 1.2.

 Theorems \ref{main thm} and \ref{mainthm0} will be proved by using the classical integral method due to Hsiung \cite{Hs} and Reilly \cite{Reilly}. The main tools are Minkowski formulae as well as a  family of Newton-Maclaurin  inequalities. Unlike in the space forms, the Newton tensor is generally not divergence-free in the warped  product manifolds.  As observed in \cite{BE}, the extra terms will have a good sign under the condition (C2) and star-shapedness.  However, to deal with our rigidity problems, one needs to keep trail with these terms carefully rather than just throw them away. On the other hand, by the generality of warped product manifolds,  the classical Alexandrov's reflection method \cite{A} as in \cite{Kor} seems to be difficult to deal with our problems.

Next we will also study  similar rigidity problems on some ``weighted'' higher mean curvatures and their linear combinations. We denote the weight in the warped product manifolds by $V(r):=\lambda'(r)$.  In \cite{W}, the first author discussed such rigidity result in $\H^n$. This kind of ``weighted'' mean curvature appears very naturally.  Interestingly, the corresponding weighted Alexandrov-Fenchel inequalities relate to the quasi-local mass in $\H^n$ and the Penrose inequalities for asymptotically hyperbolic graphs, see \cite{DGS,GWW3} for instance.   Our next result is regarding the above weighted rigidity results in the warped product manifolds.

\begin{theo}\label{mainthm2}
Let $(M^n,\bar{g})$ be an $n$-dimensional ($n\geq 3$) warped product manifold  satisfying  (C1)-(C3). Let $0\leq l<k\leq n-1$ be two integers and $\Sigma^{n-1}$ be a closed star-shaped hypersurface in $(M^n,\bar g)$.
If one of the following case holds:
\begin{itemize}
\item[(i)]  $(M,\bar g)$ satisfies (C4) and $V H_k$ is a constant for some $k=1,\cdots,n-1$;
\vspace{2mm}
\item[(ii)]$2\leq l<k\leq n-1$, $\S$ is k-convex and there are nonnegative constants  $\{a_i\}_{i=1}^{l-1}$ and $\{b_j\}_{j=l}^k$, at least one of them not vanishing,
such that $$\sum_{i=1}^{l-1}a_iH_i=\sum_{j=l}^k b_j(VH_j);$$
\item[(iii)] $1\leq l<k\leq n-2$, $\S$ is k-convex and there are nonnegative constants $\{a_i\}_{i=0}^{l-1}$ and $\{b_j\}_{j=l}^k$,  at least one of them not vanishing,
such that $$\sum_{i=0}^{l-1}a_iH_i=\sum_{j=l}^k b_j (VH_j);$$
\end{itemize}
then $\Sigma$ is a slice $N\times \{r\}$ for some $r\in (0,\bar{r})$.
\end{theo}

Theorem \ref{mainthm2} is proved in a similar way by taking the consideration of a new Minkowski type formula, Proposition \ref{prop1}. We note that  the presence of the weight makes   Alexandrov's reflection method   hard to apply even in the case of space forms, see \cite{W}.

\begin{RK}

\

\begin{itemize}
\item[(1)] Comparing with the results in \cite{Brendle, BE}, in the most cases we do not assume (C4). In fact,  we mostly will not use the Heintze-Karcher type inequality derived in \cite{Brendle}, for which (C4) is essential.
\item[(2)] Theorem \ref{mainthm0} contains the simplest case that $H_1$ is constant. In view of  Brendle's result in \cite{Brendle}, for this case, if one assumes further (C4) on $M$,  the condition of $1$-convexity and star-shapedness on hypersurfaces is actually superfluous. Similarly, the condition of star-shapedness is needless in Theorem \ref{mainthm0} when we consider $VH_1$ is a constant.
\item[(3)] Theorem \ref{mainthm0} also contain  the case that higher order mean curvatures $H_k$  are constant. For this case,  the $k$-convexity condition is superfluous since it is implied by the constancy of $H_k$.
\item[(4)] For similar rigidity problem in the space forms, the star-shapedness is not necessary. See Theorem \ref{mainthm00} below.
\end{itemize}
\end{RK}

\

The second part of this paper is about rigidity problems on some intrinsic curvature functions of induced metric from that of the  space forms. In fact, this is one of our motivations to study the linear combinations of mean curvature functions. As mentioned at the beginning, Liebmann \cite{Li} showed closed surfaces with constant Gauss curvature in $\mathbb{R}^3$ are spheres. Apparently, in space forms, one can see from the Gauss formula that surfaces with constant  scalar (Gauss) curvature is equivalent to constant $2$-nd mean curvature. Hence Liebmann's result is equivalent to Ros' \cite{Ros0}.  On the other hand, there is a natural generalization of scalar curvature, called Gauss-Bonnet curvatures. The Pfaffian in Gauss-Bonnet-Chern formula is the highest order Gauss-Bonnet curvature.  The general one  appeared first in the paper of Lanczos \cite{Lan} in 1938 and has been intensively studied in the theory of Gauss-Bonnet gravity, which is a generalization of Einstein gravity. Precisely, the Gauss-Bonnet curvatures are defined by
 \begin{equation}\label{Lk}
L_k:=\frac{1}{2^k}\d^{i_1i_2\cdots i_{2k-1}i_{2k}}
_{j_1j_2\cdots j_{2k-1}j_{2k}}{R_{i_1i_2}}^{j_1j_2}\cdots
{R_{i_{2k-1}i_{2k}}}^{j_{2k-1}j_{2k}},
\end{equation}
where $\d^{i_1i_2\cdots i_{2k-1}i_{2k}}
_{j_1j_2\cdots j_{2k-1}j_{2k}}$ is the generalized Kronecker delta defined in (\ref{generaldelta}) below and ${R_{ij}}^{kl}$ is the Riemannian curvature $4$-tensor  in local coordinates.
It is easy to see that $L_1$ is just the scalar curvature $R$. When $k=2$, it is the second Gauss-Bonnet curvature
\[L_2 = \|Rm\|^2-4\|Ric\|^2+R^2.\]
 For general $k$ it is the Euler integrand in the
Gauss-Bonnet-Chern theorem  if $n=2k$ and
 is therefore called the dimensional continued Euler density
in physics if $k<n$. Here $n$ is the dimension of corresponding manifold. Using the Gauss-Bonnet curvatures one can define the Gauss-Bonnet-Chern mass and guarantee its well-defineness in asymptotically flat manifolds as well as asymptotically hyperbolic manifolds, see \cite{GWW0,GWW1,GWW3}.

In the Euclidean space $\R^n$, the intrinsic Gauss-Bonnet curvatures  $L_k$ with the induced metric on the surfaces are the same with $H_{2k}$, up to some scaling constant. In the space forms rather than $\R^n$, $L_k$ can be expressed as some linear combination of $H_k$ (see Lemma \ref{lem0} below). Explicitly, for the unit sphere $\mathbb{S}^n$, $$L_k=\binom{n-1}{2k} (2k)!\sum_{i=0}^k \binom{k}{i} H_{2k-2i}.$$  Notice here all the coefficients are positive.
Therefore as a direct consequence of Theorem \ref{mainthm00} (ii), we have the following

\begin{coro}\label{thmsphere}Let $1\leq k\leq \frac{n-1}{2}$ be an integer and $\Sigma$ be a closed $2k$-convex hypersurface embedded in the hemisphere $\mathbb{S}_+^n$. If the $k$-th Gauss-Bonnet curvature $L_k$ is constant, then $\Sigma$ is a centered geodesic hypersphere.
\end{coro}

Unlike  in $\mathbb{S}^n$, the  intrinsic Gauss-Bonnet curvature $L_k$ in $\mathbb{H}^n$ is a linear combination of $H_k$ with sign-changed coefficients. Precisely,
$$L_k=\binom{n-1}{2k} (2k)!\sum_{i=0}^k \binom{k}{i} (-1)^i H_{2k-2i}.$$
Hence we cannot apply Theorem \ref{mainthm00} (ii) directly to conclude the rigidity. Moreover, we could prove the general rigidity result of hypersurfaces in terms of the constant linear combinations of $L_k$. This rigidity of combination form is not direct which evolves the development of some new kind Newton-Maclaurin type inequalities on $L_k$ rather than $H_k$ (see Proposition \ref{key_lemm} and Propositon \ref{key_lemm2} below ), for horoconvex hypersurfaces.  Here a hypersurface in $\H^n$ is {\it horospherical convex}
 if all its principal curvatures are larger than or equal to $1$.
 The horospherical convexity is a natural geometric concept, which is equivalent to
 the geometric convexity in Riemannian manifolds.

\begin{theo}\label{thmLk}Let $1\leq l<k\leq \frac{n-1}{2}$ be two integers and $\Sigma$ be a closed horospherical convex hypersurface in the hyperbolic space $\H^n$. If  there are nonnegative constants $\{a_i\}_{i=0}^{l-1}$ and $\{b_j\}_{j=l}^k$,  at least one of them not vanishing, such that
$$\sum_{i=0}^{l-1}a_iL_i=\sum_{j=l}^k b_j L_j,$$
 then $\Sigma$ is a centered geodesic hypersphere. In particular, if $L_k$  is constant,  then $\Sigma$ is a centered geodesic hypersphere.
\end{theo}

For $\mathbb{S}_+^n$, we can also establish similar Newton-Maclaurin type inequalities for $2k$-convex hypersurfaces, which enables us  to prove rigidity in the hemisphere $\mathbb{S}_+^n$ for a general linear combination of curvatures, as in $\mathbb{H}^n$.

\begin{theo}\label{thmLkNk}Let $1\leq l<k\leq \frac{n-1}{2}$ be two integers and $\Sigma$ be a closed $2k$-convex hypersurface embedded in the hemisphere $\mathbb{S}^n_+$. If  there are nonnegative constants $\{a_i\}_{i=0}^{l-1}$ and $\{b_j\}_{j=l}^k$,  at least one of them not vanishing, such that
$$\sum_{i=0}^{l-1}a_iL_i=\sum_{j=l}^k b_j L_j,$$
then $\Sigma$ is a centered geodesic hypersphere.
\end{theo}

Note that Theorem \ref{thmLkNk} is an extension of Corollary \ref{thmsphere}. However, it does not follow directly from Theorem \ref{mainthm00} below.

\

The paper is organized as follows. In Section 2, we provide several preliminary results including the most important tool of this paper, Minkowski type formulae. Section 3 is devoted to prove our main theorems of the first part, Theorems \ref{main thm}-\ref{mainthm2}. In Section 4, we focus on the rigidity problem on the intrinsic Gauss-Bonnet curvatures and show Theorems \ref{thmLk} and \ref{thmLkNk}.

\section{Preliminaries}
In this section,  let us  first recall some basic definitions and properties of higher order mean curvature.

Let $\s_k$ be the $k$-th elementary symmetry function $\s_k:\R^{n-1}\to \R$ defined by
\[\s_k(\Lambda)=\sum_{i_1<\cdots<i_{k}}\lambda_{i_1}\cdots\lambda_{i_k}\quad  \hbox{ for } \Lambda=(\lambda_1, \cdots,\lambda_{n-1})\in \R^{n-1}.\]
 For a symmetric $n\times n$ matrix $B$, let $\lambda(B)=(\lambda_1(B),\cdots,\lambda_n(B))$ be the real eigenvalues of $B$. We set
\[
\s_k(B):=\s_k(\lambda(B)).
\]
We denote by $$\s_k({\L}_j)=:\s_k(\l_1,\cdots,\l_{j-1},\l_{j+1},\cdots,\l_{n-1}), \hbox{ for }1\leq k\leq n-2.$$
The $k$-th Newton transformation is defined as follows
\begin{equation}\label{Newtondef}
(T_k)^{i}_{j}(B):=\frac{\partial \s_{k+1}}{\partial B^{j}_{i}}(B),
\end{equation}
where $B=(B^{i}_{j})$. We recall the basic formulas about $\s_k$ and $T$.\\
\begin{eqnarray}\label{sigmak}
\s_k(B)& =&\ds\frac{1}{k!}\d^{i_1\cdots i_k}
_{j_1\cdots j_k}B_{i_1}^{j_1}\cdots
{B_{i_k}^ {j_k}}=\frac{1}{k} \tr(T_{k-1}(B)B),
\\
(T_k)^{i}_j(B) & =& \ds \frac{1}{k!}\d^{ii_1\cdots i_{k}}
_{j j_1\cdots j_{k}}B_{i_1}^{j_1}\cdots
{B_{i_{k}}^{j_{k}}}\label{Tk}.
\end{eqnarray}
Here
 the generalized Kronecker delta is defined by
\begin{equation}\label{generaldelta}
 \d^{j_1j_2 \cdots j_r}_{i_1i_2 \cdots i_r}=\det\left(
\begin{array}{cccc}
\d^{j_1}_{i_1} & \d^{j_2}_{i_1} &\cdots &  \d^{j_r}_{i_1}\\
\d^{j_1}_{i_2} & \d^{j_2}_{i_2} &\cdots &  \d^{j_r}_{i_2}\\
\vdots & \vdots & \vdots & \vdots \\
\d^{j_1}_{i_r} & \d^{j_2}_{i_r} &\cdots &  \d^{j_r}_{i_r}
\end{array}
\right).
\end{equation}
We use the convention that $T_{-1}=0$.
 The $k$-th positive Garding cone $\Gamma_k^+$ is defined by
\begin{equation}\label{G-cone}
\Gamma_k^+=\{\Lambda \in \R^{n-1} \, |\,\s_j(\Lambda)>0, \quad\forall\, j\le k\}.
\end{equation}
And its closure is denoted by $ \overline{{\Gamma}_k^+}$.
A symmetric matrix $B$ is said to belong to $\Gamma_k^+$  if $\lambda(B)\in \Gamma_k^+$.
Let
\begin{equation}\label{Hk}
H_k=\frac{\s_k}{\binom{n-1}{k}},
\end{equation}
be the normalized $k$-th elementary symmetry function. As a convention, we take $H_0=1,\;H_{-1}=0$. The following Newton-Maclaurin inequalities are well known. For a proof, we refer to a survey of Guan \cite{Guan}.

\begin{lemm} \label{lem} For $1\leq l<k\leq n-1$ and $\Lambda\in \overline{{\Gamma}_k^+}$, the following inequalities hold:
\beq\label{NM}
H_{k-1}H_l\geq H_kH_{l-1}.
\eeq
\begin{equation}\label{N-M}
H_l\geq H_k^{\frac{l}{k}}.
\end{equation}
Moreover,  equality holds in (\ref{NM}) or (\ref{N-M}) at $\Lambda$ if and only if $\Lambda=c(1,1,\cdots,1)$ for some $c\in \mathbb{R}$.
\end{lemm}

Next, we collect some well-known results for the warped product manifold $(M=[0,\bar{r})\times_\l N(K),\bar{g})$.

We denote by $\bar{\nabla}$ and $\nabla$ the covariant derivatives on $M^n$ and the surface $\S$ respectively. As in \cite{Brendle, BE}, we define a smooth function $V: M\to \RR$ and a vector field $X$ on $M$ by $V(r)=\lambda'(r)$ and $X=\lambda(r)\frac{\p}{\p r}$. Note that $X$ is a conformal vector field satisfying
\begin{eqnarray}\label{X}
\bar{\nabla} X=V\bar{g}.
\end{eqnarray}
Condition (C1) implies that $V$ is a positive function on $(0,\bar{r})\times N(K)$. One can verify that every slice $\{r\}\times N(K)$, $r\in(0,\bar{r})$, has constant principal curvatures $\frac{\lambda'(r)}{\lambda(r)}>0.$

The Ricci curvature of $(M, \bar{g})$ is given by
\begin{eqnarray}\label{Ric}
Ric&=&-\left(\frac{\lambda''(r)}{\lambda(r)}-(n-2)\frac{K-\lambda'(r)^2}{\lambda(r)^2}\right)\bar{g}\nonumber\\&&-(n-2)\left(\frac{\lambda''(r)}{\lambda(r)}+\frac{K-\lambda'(r)^2}{\lambda(r)^2}\right)dr\otimes dr.
\end{eqnarray}

Let $\{e_i\}_{i=1}^{n-1}$ and $\nu$ be an orthonormal basis and the outward normal of $\S$ respectively. Denote by $h_{ij}$ the second fundamental form of $\S$ with this basis and $\l=(\l_1,\cdots,\l_{n-1})$  the principal curvatures of $\S$.
 The star-shapedness of $\S$ means
\beq\label{starshape}
\<\frac{\p}{\p r},\nu\>\geq 0.
\eeq

We need the following Minkowski type formula in the product manifolds, which is included in the proof of \cite{BE}, Proposition 8 and Proposition 9. For completeness, we involve a proof here.

\begin{prop}\label{Minkowskiformula}
Let $\Sigma^{n-1}$ be a closed hypersurface isometric immersed in the product manifold $(M,\bar{g})$. Then
\begin{itemize}
\item[(i)] we have
\begin{eqnarray}\label{0eq0}
\int_\S \<X,\nu\>H_kd\mu=\int_\S VH_{k\!-\!1}d\mu\!+\frac{k\!-\!1}{\binom{n-2}{k-2}}\int_\S \sum_{i,j=1}^{n-1}A_{ij}(T_{k-2})_{ij}d\mu,\;\, \forall\;  1\leq k\leq n\!-\!1,
\end{eqnarray}
where
\beq\label{Aij}
A_{ij}:=-\frac{1}{(n-1)(n-2)}\<X,e_i\>Ric(e_j,\nu).
\eeq
\vspace{2mm}
\item[(ii)] If $\S$ is star-shaped and $(M,\bar g)$ satisfies (C2), then we have
\beq\label{Aj}
A_{jj}\geq 0,\;\, \forall \; 1\leq j\leq n-1.
\eeq
\end{itemize}
\end{prop}

\begin{proof} (i)  It follows from the Gauss-Weingarten formula and \eqref{X} that
\begin{eqnarray}\label{0eq1}
\nabla_i X_j=\bar{\nabla}_iX_j-\<X,\nu\>h_{ij}=V\bar{g}_{ij}-\<X,\nu\>h_{ij}.\end{eqnarray}
Multiplying \eqref{0eq1} by the $k$-th Newton transform tensor $(T_{k-1})_{ij}$ and summing over $i, j$, we obtain
\begin{eqnarray}\label{0eq2}
\sum_{i,j=1}^{n-1}\nabla_i \left(X_j(T_{k-1})_{ij}\right)&=&\sum_{j=1}^{n-1}X_j \sum_{i=1}^{n-1}\nabla_i (T_{k-1})_{ij} +V\sum_{i,j=1}^{n-1}(T_{k-1})_{ij}\bar{g}_{ij}-\<X,\nu\>\sum_{i,j=1}^{n-1}(T_{k-1})_{ij}h_{ij}\nonumber\\&=&\sum_{j=1}^{n-1}X_j\sum_{i=1}^{n-1} \nabla_i (T_{k-1})_{ij}+V(n-k)\s_{k-1}-k\s_{k}\<X,\nu\>,\end{eqnarray}
where (\ref{sigmak}) and (\ref{Tk}) are used to get (\ref{0eq2}).

By the definition of $(T_{k-1})_{ij},$ we know that
\begin{eqnarray}\label{0eq3}
&&\sum_{i=1}^{n-1}\nabla_i (T_{k-1})_{ij}=\sum_{i=1}^{n-1}\sum_{\substack{i_1,\cdots, i_{k-1}=1,\\ j_1,\cdots, j_{k-1}=1}}^{n-1}\frac{1}{(k-2)!}\delta_{j_1\cdots j_{k-1} j}^{i_1\cdots i_{k-1} i}\nabla_i h_{i_1j_1}\cdots h_{i_{k-1}j_{k-1}}\nonumber\\&=&\sum_{i=1}^{n-1}\sum_{\substack{i_1,\cdots, i_{k-1}=1,\\ j_1,\cdots, j_{k-1}=1}}^{n-1}\frac{1}{2}\frac{1}{(k-2)!}\delta_{j_1\cdots j_{k-1} j}^{i_1\cdots i_{k-1} i}\left(\nabla_i h_{i_1j_1}-\nabla_{i_1} h_{ij_1}\right)h_{i_2j_2}\cdots h_{i_{k-1}j_{k-1}}.
\end{eqnarray}

As $N(K)$ is of constant sectional curvature, it is easy to see that $M$ is locally conformally flat.
Using Codazzi equation and the local conformal flatness of $M$, we have
\begin{eqnarray}\label{0eq4}
\nabla_i h_{i_1j_1}-\nabla_{i_1} h_{ij_1}&=& Riem(e_i,e_{i_1},e_{j_1},\nu)\nonumber\\&=&-\frac{1}{n-2}\left(Ric(e_i,\nu)\delta_{i_1j_1}-Ric(e_{i_1},\nu)\delta_{ij_1}\right).
\end{eqnarray}
Substituting \eqref{0eq3} into \eqref{0eq4}, we deduce that
\begin{eqnarray}\label{0eq5}
\sum_{i=1}^{n-1}\nabla_i (T_{k-1})_{ij}&=&-\frac{1}{n-2}\sum_{i=1}^{n-1}Ric(e_i,\nu)\sum_{\substack{i_2,\cdots, i_{k-1}=1,\\ j_1,\cdots, j_{k-1}=1,\\  j_1\neq j, j_2,\cdots, j_{k-1}}}^{n-1} \frac{1}{(k-2)!}\delta_{j_1j_2\cdots j_{k-1} j}^{j_1i_2\cdots i_{k-1} i}h_{i_2j_2}\cdots h_{i_{k-1}j_{k-1}}\nonumber\\&=& -\frac{n-k}{n-2}\sum_{i=1}^{n-1}Ric(e_i,\nu) \sum_{\substack{i_2,\cdots, i_{k-1}=1,\\ j_2,\cdots, j_{k-1}=1}}^{n-1}\frac{1}{(k-2)!} \delta_{j_2\cdots j_{k-1} j}^{i_2\cdots i_{k-1} i}h_{i_2j_2}\cdots h_{i_{k-1}j_{k-1}} \nonumber\\&=&-\frac{n-k}{n-2}\sum_{i=1}^{n-1}Ric(e_i,\nu)(T_{k-2})_{ij}.
\end{eqnarray}

Now by taking integration of \eqref{0eq2} over $\S$ together with \eqref{0eq5} and taking (\ref{Hk}) into account, we arrive at \eqref{0eq0}.

(ii) We know from \eqref{Ric} that $$Ric(e_j,\nu)=-(n-2)\left(\frac{\lambda''(r)}{\lambda(r)}+\frac{K-\lambda'(r)^2}{\lambda(r)^2}\right)\frac{1}{\lambda(r)^2}
\<X,e_j\>\<X,\nu\>,$$
which implies
\begin{eqnarray*}
A_{jj}&=&-\frac{1}{(n-1)(n-2)}\<X,e_j\>Ric(e_j,\nu)\\
&=& \frac{1}{(n-1)} \left(\frac{\lambda''(r)}{\lambda(r)}+\frac{K-\lambda'(r)^2}{\lambda(r)^2}\right)\frac{1}{\lambda(r)^2}
\<X,e_j\>^2\<X,\nu\>.
\end{eqnarray*}
By using the star-shapedness (\ref{starshape}) of $\S$ and the assumption (C2) on $\lambda(r)$, we conclude $A_{jj}\geq 0$ for any $j=1,\cdots, n-1$.
\end{proof}

For later purpose to prove the rigidity result on weighted curvature functions, we need to extend the above proposition to the following type.
\begin{prop}\label{prop1}
Let $\Sigma$ be a hypersurface isometric immersed in the product manifold $(M^n,\bar g)$, we have
\begin{equation}\label{eq2}
\int_{\Sigma}\<X,\nu\> VH_kd\mu= \int_{\Sigma}V^2H_{k-1}d\mu+\frac{k\!-\!1}{\binom{n-2}{k-2}}\int_\S \sum_{i,j=1}^{n-1}VA_{ij}(T_{k-2})_{ij}d\mu+\frac{1}{k\binom{n-1}{k}}\int_{\Sigma}(T_{k-1})_{ij}X_i \nabla_j Vd\mu.
\end{equation}
Moreover, if $\S$ is k-convex and $(M,\bar g)$ satisfies condition (C3), then we have
\begin{equation}\label{eq1}
\int_{\Sigma}\<X,\nu\>VH_kd\mu\geq \int_{\Sigma}V^2H_{k-1}d\mu+\frac{k\!-\!1}{\binom{n-2}{k-2}}\int_\S \sum_{i,j=1}^{n-1}VA_{ij}(T_{k-2})_{ij}d\mu.
\end{equation}
Equality holds if and only if $\Sigma$ is totally umbilical in $(M^n,\bar g)$.
\end{prop}
\begin{proof}
Combining (\ref{0eq2}) and (\ref{0eq5}) together, we arrive at
\beq
\frac{1}{k\binom{n-1}{k}}\nabla_i \left((T_{k-1})^{ij}X_j\right)=-\<X,\nu\>H_k+VH_{k-1}+\frac{k\!-\!1}{\binom{n-2}{k-2}}\sum_{i,j=1}^{n-1}A_{ij}(T_{k-2})_{ij},
\eeq
where $A_{ij}$ is defined in (\ref{Aij}).
Multiplying above  equation by the function $V$ and integrating by parts, one obtains the desired result (\ref{eq2}).
Noting that
$$X_i=\lambda(r)\nabla_i r,\;\; \nabla_j V=\lambda''(r)\nabla_j r,$$
we have
\beq\label{T}
(T_{k-1})^{ij}X_i \nabla_j V=\lambda(r)\lambda''(r)(T_{k-1})^{ij}\nabla_i r\nabla_j r.
\eeq
Under the assumption that $\Sigma$ is  $k$-convex,  the $(k-1)$-th Newton tensor  $T_{k-1}$ is positively definite (see e.g. Guan \cite{Guan}), hence
$$(T_{k-1})^{ij}\nabla_i r\nabla_j r\geq 0.$$
Together with assumption (C3)  $\lambda''(r)\geq 0$, (\ref{eq1}) holds. When the equality holds, we have $\nabla r=0$ which implies that $\Sigma$ is umbilical in $(M^n,\bar g).$
\end{proof}

Finally, we need  a  Heintze-Karcher-type
inequality due to Ros \cite{Ros} and Brendle \cite {Brendle}.
\begin{prop}[Brendle]\label{B}
Let $(M^n=[0,\bar{r})\times N(K), \bar{g}=dr^2+\lambda(r)^2 g_N)$ be a warped product  space satisfying (C1),(C2),(C4), or one of the space forms $\mathbb{R}^n,$ $ \mathbb{S}^n_{+}$,  $\mathbb{H}^n$. Let $\Sigma$ be a compact hypersurface embedded in $(M^n,\bar g)$ with positive mean curvature $H_1$, then
$$\int_{\Sigma}\<X,\nu\>d\mu\leq \int_{\Sigma} \frac{V}{H_1}d\mu.$$
Moreover, equality holds if and only if $\Sigma$ is totally umbilical.
\end{prop}

\section{Rigidity for curvature quotients and combinations}

In this section, we are ready to prove our main theorems. We start with the one on curvature quotients.
This will be proved by making use of Lemma \ref{lem} and Proposition \ref{Minkowskiformula}.

\vspace{2mm}
\noindent\textit{Proof of Theorem \ref{main thm}:}
We first claim that $\l\in \Gamma^+_k$. In fact, condition (C1) implies that $\S$ has at least one elliptic point where all the principal curvatures are positive. This can be shown by a standard argument using maximum principle. Hence the constant $c$ should be positive. Moreover, since $H_l$ is nowhere vanishing on $\S$, it must be positive. In turn, $H_k=cH_l$ is positive. From the result of G\r{a}rding \cite{Garding}, we know that $H_j>0$ everywhere on $\S$ for  $1\leq j\leq k$.

For $1\leq l<k\leq n-1$, Proposition \ref{Minkowskiformula} gives the following two formulae:
\begin{eqnarray}\label{mink}
\int_\S \<X,\nu\>H_kd\mu&=&\int_\S VH_{k-1}d\mu+\frac{k\!-\!1}{\binom{n-2}{k-2}}\int_\S \sum_{i,j=1}^{n-1}A_{ij}(T_{k-2})_{ij}d\mu,\\
\int_\S \<X,\nu\>H_l d\mu&=&\int_\S VH_{l-1}d\mu+\frac{l\!-\!1}{\binom{n-2}{l-2}}\int_\S \sum_{i,j=1}^{n-1}A_{ij}(T_{l-2})_{ij}d\mu.\label{minl}
\end{eqnarray}

Since $H_k=cH_l$, we deduce from \eqref{mink},\eqref{minl} together with (\ref{NM}), (\ref{Aj}) that
\begin{eqnarray}\label{eq3_thm1.1}
0&=&\int_\S \<X,\nu\>(H_k-cH_l)d\mu\nonumber\\&=&\int_\S V\left(H_{k-1}-cH_{l-1}\right)d\mu+\int_\S\sum_{i,j=1}^{n-1} A_{ij}\left(\frac{k\!-\!1}{\binom{n-2}{k-2}}(T_{k-2})_{ij}-c\frac{l\!-\!1}{\binom{n-2}{l-2}}(T_{l-2})_{ij}\right)d\mu.
\end{eqnarray}
Without loss of generality, one may assume that the second fundamental form $h_{ij}$ is diagonal at the point under computation.
At this point, we have
\begin{eqnarray}\label{point1}
&&\sum_{i,j=1}^{n-1}A_{ij}\left(\frac{k\!-\!1}{\binom{n-2}{k-2}}(T_{k-2})_{ij}-c\frac{l\!-\!1}{\binom{n-2}{l-2}}(T_{l-2})_{ij}\right)\nonumber\\
&=&\sum_{j=1}^{n-1}A_{jj}\left((k-1)H_{k-2}(\L_j)-c(l-1)H_{l-2}(\L_j)\right).
\end{eqnarray}

We know  from the Newton-Maclaurin inequality (\ref{NM}) that
\beq\label{eq0110}
\frac{H_{k-1}}{H_{l-1}}\geq \frac{H_k}{H_l}=c.
\eeq
On the other hand, note the simple fact
$$\sigma_k=\lambda_j\sigma_{k-1}(\L_j)+\s_{k}(\L_j),$$
which is equivalent to
\beq\label{HLambda}
H_k=\frac{k}{n-1}\l_jH_{k-1}(\L_j)+\frac{n-1-k}{n-1}H_k(\L_j).
\eeq
Applying (\ref{HLambda}), for any $j=1,\cdots, n-1$, we find
\begin{eqnarray}\label{HH}
&&(k-1)H_{k-2}(\L_j)H_{l-1}-(l-1)H_{k-1}H_{l-2}(\L_j)\nonumber\\
&=&\frac{(k-1)(n-l)}{n-1}H_{k-2}(\L_j)H_{l-1}(\L_j)-\frac{(l-1)(n-k)}{n-1}H_{l-2}(\L_j)H_{k-1}(\L_j)\nonumber\\
&=&(k\!-\!l)H_{k-2}(\L_j)H_{l-1}(\L_j)\!+\!\frac{(l\!-\!1)(n\!-\!k)}{n-1}\left(H_{k-2}(\L_j)H_{l-1}(\L_j)\!-\!H_{l-2}(\L_j)H_{k-1}(\L_j)\right)\nonumber\\
&>&0.
\end{eqnarray}

Therefore, by \eqref{Aj}, \eqref{point1}, \eqref{eq0110} and \eqref{HH}, the integrand in the right hand side of (\ref{eq3_thm1.1}) is non-negative. It follows that the equality holds in (\ref{eq0110}), which implies that $\S$ is totally umbilical. Moreover, thanks to \eqref{HH}, we have
\begin{eqnarray}\label{AA}
A_{jj}\equiv 0,\quad \forall\; 1\leq j\leq n-1.
\end{eqnarray}
 Together with condition (C2), \eqref{AA} implies that the normal $\nu$ is parallel or pendicular to $\frac{\p}{\p r}$ everywhere on $\S$. However, there is at least one point on $\S$ where $\nu$ is parallel to $\frac{\p}{\p r}$. Therefore, $\nu$ is parallel to $\frac{\p}{\p r}$ for all points in $\S$, which means that  $\S$ is a slice $\{r\}\times  N(K)$. We complete the proof.\qed

\

Next we show the rigidity result for constant linear combinations of mean curvatures in the warped product manifolds. This argument basically follows from the above one except that one needs pay more attention to the use of the Newton-Maclaurin inequality at the first step.

\vspace{2mm}
\noindent {\it Proof of  Theorem \ref{mainthm0}:}

(i) By the existence of an elliptic point and non-vanishing of at least one coefficient, we know $\sum_{i=1}^{l-1}a_i H_i >0.$
Since $\Sigma$ is k-convex, we recall from (\ref{NM}) that
\begin{eqnarray}\label{ab1}
H_i H_{j-1}\geq H_{i-1}H_j,\; 1\leq i<j\leq k,
\end{eqnarray}
where all equalities hold if and only if $\S$ is umbilical.
Multiplying \eqref{ab1} by $a_i$ and $b_j$  and summing over $i$ and $j$, we get
\begin{eqnarray}\label{ab2}
\sum_{i=1}^{l-1}a_i H_i  \sum_{j=l}^k b_j H_{j-1}\geq \sum_{i=1}^{l-1} a_i H_{i-1} \sum_{j=l}^k  b_jH_j.
\end{eqnarray}
By using the assumption $$\sum_{i=1}^{l-1}a_i H_i=\sum_{j=l}^k  b_jH_j>0,\;\,2\leq l<k\leq n-1,$$ we obtain from \eqref{ab2} that
\begin{eqnarray}\label{ab3}
 \sum_{j=l}^k b_j H_{j-1}\geq \sum_{i=1}^{l-1} a_i H_{i-1}.
\end{eqnarray}

On the other hand, (\ref{NM}) and (\ref{HH}) give
\begin{eqnarray}\label{ab4}
(j-1)H_{j-2}(\L_p) H_i>   (i-1)H_j H_{i-2}(\L_p),\;\, \forall\; 1\leq i<j\leq k, 1\leq p\leq n-1.
\end{eqnarray}
Multiplying \eqref{ab4} by $a_i$ and $b_j$  and summing over $i$ and $j$, we have
\begin{eqnarray*}
 \sum_{j=l}^k (j-1)b_j  H_{j-2}(\L_p) \sum_{i=1}^{l-1} a_i H_i>  \sum_{j=l}^k b_j H_j \sum_{i=1}^{l-1} (i-1)a_i H_{i-2}(\L_p).
\end{eqnarray*}
Hence
\begin{eqnarray}\label{ab5}
 \sum_{j=l}^k (j-1) b_j H_{j-2}(\L_p)> \sum_{i=1}^{l-1} (i-1) a_i H_{i-2}(\L_p),\;\,\forall\;  1\leq p\leq n-1.
\end{eqnarray}
As in the proof of Theorem \ref{main thm}, \eqref{ab5} implies the matrix
\begin{eqnarray}\label{aabb5}
 \left(\sum_{j=l}^{k} \frac{(j-1)}{\binom{n-2}{j-2}} b_j (T_{j-2})_{pq}-\sum_{i=1}^{l-1} \frac{(i-1)}{\binom{n-2}{i-2}} a_i (T_{i-2})_{pq}\right)_{p,q=1}^{n-1}\hbox{ is positive definite.}
\end{eqnarray}

We finally infer from \eqref{mink}, \eqref{minl}  that
\begin{eqnarray}\label{ab6}
0&=&\int_\Sigma (\sum_{j=l}^k b_j H_j-\sum_{i=1}^{l-1} a_i H_i) \<X,\nu\> d\mu=\int_\Sigma (\sum_{j=l}^k b_j H_{j-1}-\sum_{i=1}^{l-1} a_i H_{i-1})V d\mu\nonumber\\&&+\int_\Sigma \sum_{p,q=1}^{n-1} A_{pq}\left(   \sum_{j=l}^{k} \frac{(j-1)}{\binom{n-2}{j-2}} b_j (T_{j-2})_{pq}-\sum_{i=1}^{l-1} \frac{(i-1)}{\binom{n-2}{i-2}} a_i (T_{i-2})_{pq}\right)d\mu\geq 0.
\end{eqnarray}
Here the last inequality follows from (\ref{Aj}), (\ref{point1}) (\ref{ab3}) and (\ref{aabb5}).

\vspace{2mm}
(ii) The proof is essentially the same as above. One only needs to notice the slight difference regarding the value of indices. Proceeding as above, we have

\begin{eqnarray}\label{ab3'}
 \sum_{j=l}^k b_j H_{j+1}\leq \sum_{i=0}^{l-1} a_i H_{i+1},
\end{eqnarray}
and
\begin{eqnarray}\label{ab5'}
 \sum_{j=l}^k j b_j H_{j-1}(\L_p)> \sum_{i=0}^{l-1} i a_i H_{i-1}(\L_p),\;\,\forall\;  1\leq p\leq n-1.
\end{eqnarray}
Applying (\ref{0eq0}) again,
\begin{eqnarray}\label{ab6'}
0&=&\int_\Sigma (\sum_{i=0}^{l-1} a_i H_i-\sum_{j=l}^k b_j H_j)V d\mu=\int_\Sigma (\sum_{i=0}^{l-1} a_i H_{i+1}-\sum_{j=l}^k b_j H_{j+1})\<X,\nu\> d\mu\nonumber\\&&+\int_\Sigma \sum_{p,q=1}^{n-1} A_{pq}\left(  \sum_{j=l}^k \frac{j b_j}{\binom{n-2}{j-1}} (T_{j-1})_{pq}- \sum_{i=0}^{l-1} \frac{i a_i} {\binom{n-2}{i-1}}(T_{i-1})_{pq}\right)d\mu\geq 0.
\end{eqnarray}
Here the last inequality follows from   (\ref{Aj}), (\ref{point1}), (\ref{ab3'}) and (\ref{ab5'}).

We finish the proof by examining the equality in both cases as in the proof  in Theorem \ref{main thm}.
\qed

\

As remarked in the introduction, for the same rigidity problem in the space forms, the star-shapedness is not necessary. That is, we have the following theorem.
\begin{theo}\label{mainthm00}
Let $0\leq k\leq n-1$ be an integer and $\Sigma^{n-1}$ be a closed, k-convex hypersurface in $\mathbb{R}^n (\mathbb{S}_+^n, \mathbb{H}^n,\hbox{ resp.})$.
If either of the following case holds:
\begin{itemize}
\item[(i)] $2\leq l<k\leq n-1$ and there are nonnegative constants $\{a_i\}_{i=1}^{l-1}$ and $\{b_j\}_{j=l}^k$,  at least one of them not vanishing, such that $$\sum_{i=1}^{l-1}a_iH_i=\sum_{j=l}^k b_j H_j;$$
\item[(ii)]  there are nonnegative constants $a_0$ and $\{b_j\}_{j=1}^k$, at least one of them not vanishing, such that $$a_0=\sum_{j=1}^{k}b_jH_j;$$
\end{itemize}
then $\Sigma$ is a geodesic hypersphere.
\end{theo}
For the proof of Theorem \ref{mainthm00}, we still apply the integral technique following \cite{Reilly}. We remark that it could be also obtained by using the classical Alexandrov's reflection method as in \cite{Kor}.

For the space forms $\mathbb{R}^n$ ($\mathbb{S}_+^n$, $\mathbb{H}^n$ resp.), the conformal vector field $X=r\frac{\p}{\p r}$ ($\sin r\frac{\p}{\p r}$, $\sinh r\frac{\p}{\p r}$ resp.) and $V=1$ ($\cos r$, $\cosh r$ resp.). It follows from the Codazzi equation that the Newton tensor $T_k$ is divergence-free with the induced metric on $\S$, i.e., $\nabla_i T_k^{ij}=0$. Thus (\ref{0eq2}) implies
\begin{equation}\label{Minkowski}
\nabla_j(T_{k-1}^{ij}X_i)=-k\<X,\nu\>\sigma_{k}+\left(n-k\right) V\sigma_{k-1}.
\end{equation}
Integrating above equation and noting (\ref{Hk}), we have the Minkowski formula in the space forms
\begin{equation}\label{Minkowski_Identity}
\int_{\Sigma} \<X,\nu\> H_{k}d\mu= \int_{\Sigma}V H_{k-1} d\mu.
\end{equation}

\vspace{6mm}
\noindent\textit{Proof of Theorem \ref{mainthm00}:}
(i)It follows from \eqref{Minkowski_Identity} and (\ref{NM}) that
\begin{eqnarray}
0=\int_\S \<X,\nu\>(\sum_{i=1}^{l-1}a_iH_i- \sum_{j=l}^k b_j H_j)d\mu= \int_\S V(\sum_{i=1}^{l-1}a_iH_{i-1}- \sum_{j=l}^k b_j H_{j-1})d\mu\leq 0.
\end{eqnarray}
The last inequality follows from \eqref{ab3}, where equality holds if and only if $\S$ is a geodesic hypersphere.

(ii) From the existence of an elliptic point and non-vanishing of at least one coefficient, we have $\sum_{j=1}^k b_jH_j>0$. Since $\Sigma$ is $k$-convex, $$\sum_{j=1}^k b_j H_1^j\geq \sum_{j=1}^k b_jH_j>0.$$ Hence $H_1$ cannot vanish at any points, which implies that $H_1>0$. Making use of (\ref{Minkowski_Identity}) and (\ref{NM}), we derive
\begin{eqnarray*}
a_0\int_\S \<X,\nu\> d\mu&=&\int_\S \<X,\nu\>  \left(\sum_{j=1}^{k}b_jH_j\right)d\mu = \int_\S V\left(\sum_{j=1}^{k}b_jH_{j-1}\right)\frac{H_1}{H_1}d\mu\\&\geq  &\int_\S V\left(\sum_{j=1}^{k}b_jH_j\right)\frac{1}{H_1}d\mu=a_0\int_\S \frac{V}{H_1}d\mu\\
&\geq& a_0\int_\S \<X,\nu\> d\mu,
\end{eqnarray*}
where in the last inequality we used Proposition \ref{B}.
Therefore, the equality in both case yields that $\S$ is a geodesic hypersphere.
\qed

\

Using a similar argument and taking Propositions \ref{prop1} and \ref{B} into account, we now prove the rigidity for the weighted curvature functions.

\vspace{2mm}
\noindent{\it Proof of  Theorem \ref{mainthm2}:}
(i)
First the existence of an elliptic point implies that $H_k$ is positive everywhere on $\S$. Then we know that $H_j>0$ and $H_{j}(\L_p)\geq 0,\;\,\forall\;1\leq p\leq n-1$ for $1\leq j\leq k$.

Thus (\ref{eq1}) implies
\begin{equation}\label{eq3}
\int_{\Sigma}\<X,\nu\>VH_kd\mu\geq \int_{\Sigma}V^2H_{k-1}d\mu.
\end{equation}
Noticing from (\ref{N-M}) that
$$H_{k-1}\geq H_k^{\frac{k-1}{k}},$$
we compute
\begin{eqnarray*}
V H_k\int_{\Sigma} \<X,\nu\> d\mu&=&\int_{\Sigma} \<X,\nu\> VH_kd\mu\geq \int_{\Sigma}V^2H_{k-1}d\mu\\
&\geq& \int_{\Sigma} V^2 H_k^{\frac{k-1}{k}}d\mu=(VH_k)^{\frac{k-1}{k}}\int_{\Sigma}V^{1+\frac 1k}d\mu,
\end{eqnarray*}
which yields
\begin{equation}\label{eq5}
\int_{\Sigma}\<X,\nu\>d\mu\geq (VH_k)^{-\frac 1k}\int_{\Sigma}V^{1+\frac 1k}d\mu,
\end{equation}
and  equality holds if and only $\Sigma$ is a  geodesic sphere.

On the other hand, by Proposition \ref{B} and (\ref{N-M}) we derive that
\begin{equation}\label{eq6}
\int_{\Sigma}\<X,\nu\>d\mu\leq \int_{\Sigma} \frac{V}{H_1}d\mu\leq \int_{\Sigma}\frac{V}{H_k^{\frac 1k}}d\mu=(VH_k)^{-\frac 1k}\int_{\Sigma} V^{1+\frac 1k}d\mu.
\end{equation}
Finally combining (\ref{eq5}) and (\ref{eq6}) together, we complete the proof.

\vspace{2mm}
(ii)
As in the proof of Theorem \ref{mainthm0}, one can obtain the following two inequalities: \begin{eqnarray}\label{ab33}
 \sum_{j=l}^k b_j \left(VH_{j-1}\right)\geq \sum_{i=1}^{l-1} a_i H_{i-1},
\end{eqnarray}
and
\begin{eqnarray}\label{ab55}
 \sum_{j=l}^k (j-1) b_j V H_{j-2}(\L_p)> \sum_{i=1}^{l-1} (i-1) a_i H_{i-2}(\L_p),\forall 1\leq p\leq n-1.
\end{eqnarray}

For $1\leq l<k\leq n-1$, it follows from Proposition  \ref{Minkowskiformula} and Proposition \ref{prop1} that
\begin{eqnarray}\label{mink'}
\int_\S \<X,\nu\>VH_kd\mu&\geq&\int_\S V^2H_{k-1}d\mu+\frac{k-1}{\binom{n-2}{k-2}}\int_\S \sum_{p,q=1}^{n-1}VA_{pq} (T_{k-2})_{pq}d\mu,\\
\int_\S \<X,\nu\>H_l d\mu&=&\int_\S VH_{l-1}d\mu+\frac{l-1}{\binom{n-2}{l-2}}\int_\S \sum_{p,q=1}^{n-1}VA_{pq}(T_{l-2})_{pq}d\mu.\label{minl'}
\end{eqnarray}

We then derive from above that
\begin{eqnarray}
0&=&\int_\Sigma (\sum_{j=l}^k b_j VH_j-\sum_{i=1}^{l-1} a_i H_i) \<X,\nu\> d\mu=\int_\Sigma V(\sum_{j=l}^k b_j V H_{j-1}-\sum_{i=1}^{l-1} a_i H_{i-1}) d\mu\\&&+\int_\Sigma \sum_{p,q=1}^{n-1} A_{pq}\left(   \sum_{j=l}^{k} \frac{(j-1)}{\binom{n-2}{j-2}} b_j V(T_{j-2})_{pq}-\sum_{i=1}^{l-1} \frac{(i-1)}{\binom{n-2}{i-2}} a_i (T_{i-2})_{pq}\right)d\mu\geq 0.
\end{eqnarray}
Here the last inequality follows from (\ref{Aj}), (\ref{point1}), (\ref{ab33}) and (\ref{ab55}). We finish the proof by examining the equality case as before.

\vspace{2mm}
(iii) The proof is similar with above with some necessary adaption as the one  did in the proof of Theorem \ref{mainthm0} (ii).
\qed

\

\section{rigidity for $L_k$ curvatures and their combinations}

Unlike the mean curvatures $H_k$, the Gauss-Bonnet curvatures $L_k$, and hence $\int_\Sigma L_k d\mu$ are intrinsic geometric quantities, which depend only on the induced metric  on $\Sigma$ and
are independent of the embeddings of $\Sigma$.
The  functionals $\int_\Sigma L_k$ are new geometric quantities for  the study of the integral geometry in the space forms.

We first infer a relation between $L_k$ and $H_k$.
\begin{lemm}
\label{lem0} For a hypersurface  $(\Sigma,g)$ in the space forms $\H^n$ ($\R^n$, $\mathbb{S}^n$, resp.) with constant curvature $\epsilon=-1 (0,1, resp.)$, its Gauss-Bonnet curvature $L_k$ with respect to $g$ can be expressed by higher order mean curvatures
\begin{eqnarray}\label{Lk1}
L_k&=&\binom{n-1}{2k}(2k)!\sum_{i=0}^k\binom{k}{i}\epsilon^iH_{2k-2i}.
\end{eqnarray}
\end{lemm}
\begin {proof}
 First by the Gauss formula $${R_{ij}}^{kl}=({h_i}^k{h_j}^l-{h_i}^l{h_j}^k)+\epsilon({\delta_i}^k{\delta_j}^l-{\delta_i}^l{\delta_j}^k),$$ where ${h_i}^j:=g^{ik}h_{kj}$ and $h$ is the second fundamental form.
Then substituting the Gauss formula above into (\ref{Lk}) and noting (\ref{sigmak}), a straightforward calculation leads to,
\begin{eqnarray*}
L_k&=&\frac{1}{2^k}\d^{i_1i_2\cdots i_{2k-1}i_{2k}}
_{j_1j_2\cdots j_{2k-1}j_{2k}}{R_{i_1i_2}}^{j_1j_2}\cdots
{R_{i_{2k-1}i_{2k}}}^{j_{2k-1}j_{2k}}\\
&=&\d^{i_1i_2\cdots i_{2k-1}i_{2k}}
_{j_1j_2\cdots j_{2k-1}j_{2k}}({h_{i_1}}^{j_1}{h_{i_2}}^{j_2}+\epsilon{\delta_{i_1}}^{j_1}{\delta_{i_2}}^{j_2})\cdots ({h_{i_{2k-1}}}^{j_{2k-1}}{h_{i_{2k}}}^{j_{2k}}+\epsilon{\delta_{i_{2k-1}}}^{j_{2k-1}}{\delta_{i_{2k}}}^{j_{2k}})\\
&=&\sum_{i=0}^k \binom{k}{i} \epsilon^i (n-2k)(n-2k+1)\cdots(n-1-2k+2i)\big((2k-2i)!\s_{2k-2i}\big)\\
&=&\binom{n-1}{2k}(2k)!\sum_{i=0}^k\binom{n-1}{i} \epsilon^iH_{2k-2i}.
\end{eqnarray*}
Here in the second equality we used the symmetry of generalized Kronecker delta and  in the third equality we used (\ref{sigmak}) and the basic property of generalized Kronecker delta
\beq\label{Kroneckerpro.}
\d^{i_1i_2\cdots i_{p-1}i_{p}}
_{j_1j_2\cdots j_{p-1}j_{p}}{\d_{i_1}}^{j_1}=(n-p)\d^{i_2i_3\cdots i_p}
_{j_2j_3\cdots j_p}.
\eeq
\end{proof}

Motivated by the expression (\ref{Lk1}), we introduce the following notations,
\begin{equation}\label{tildeL1}
\widetilde{L}_{k}:=\frac{L_k}{\binom{n-1}{2k}(2k!)},\qquad\widetilde{N}_{k}:=\frac{N_k}{\binom{n-1}{2k}(2k!)},
\end{equation}
where $$N_k:=\binom{n-1}{2k}(2k)!\sum_{i=0}^k\binom{k}{i}\epsilon^iH_{2k-2i+1}.$$

Since for the sphere $\mathbb{S}^n$, $L_k$ can be expressed as linear combinations of $H_k$ with nonnegative coefficients in the formula \eqref{Lk1}, thus rigidity for $L_k$ is an immediate consequence of Theorem \ref{mainthm00}.

\vspace{2mm}
\noindent{\it Proof of Corollary \ref{thmsphere}:}
In the hemisphere $\mathbb{S}_+^n$, there exists an elliptic point. Thus $L_k=const.$ is equivalent to
$$\sum_{i=1}^k\binom{k}{i}H_{2i}=a_0,$$ for some positive $a_0$. Hence the conclusion follows from Theorem \ref{mainthm00}.
\qed

\

However, the hyperbolic case is not that easy. We will apply a new kind of Newton-Maclaurin type inequality to the hyperbolic case. It is clear that in hyperbolic space
\begin{equation}\label{widetildeL}
\widetilde{L}_{k}=\sum_{i=0}^k\binom{k}{i}(-1)^{k-i}H_{2i},\qquad\widetilde{N}_{k}=\sum_{i=0}^k\binom{k}{i}(-1)^{k-i}H_{2i+1}.
\end{equation}

Due to the sign-changed coefficients of $L_k$ in terms of $H_k$, it seems to be difficult to apply Newton-Maclaurin inequalities directly.
Fortunately, under the condition of horoconvexity, we have the following refined Newton-Maclaurin inequalities \cite{GWW2}.
\begin{prop}\label{key_lemm}
 For  any $\kappa$ satisfying
\begin{equation*}\label{h-convex}
\kappa\in \{\kappa=(\kappa_1,\kappa_2,\cdots,\kappa_{n-1})\in\R^{n-1}\,|\, \kappa_i\ge 1\},
\end{equation*} we have
\begin{equation}\label{key inequ.}
\widetilde{N}_k-H_1\widetilde{L}_k\leq0.
\end{equation}
Equality holds if and only if
one of the following two cases holds
$$
\hbox{either } \quad (i)\,  \kappa_i=\kappa_j \; \forall \, i,j,  \quad\hbox{ or }\quad (ii)\,k\geq 2, \; \exists \, i\, \hbox{ with }\kappa_i >1\, \& \, \kappa_j=1 \; \forall j\neq i.$$
\end{prop}
\begin{proof}
This proposition is proved in \cite{GWW2}. The key point is to observe that (\ref{key inequ.}) is equivalent to the following inequality:
\begin{align}\label{permu-sum-3}
\sum_{1\le i_m\le n-1, i_j\neq i_l(j\neq l)} \kappa_{i_1}(\kappa_{i_2}\kappa_{i_3}-1)
(\kappa_{i_4}\kappa_{i_5}-1)\cdots(\kappa_{i_{2k-2}}\kappa_{i_{2k-1}}-1)\big(\kappa_{i_{2k}}-\kappa_{i_{2k+1}}\big)^2 \ge 0,
\end{align}
where the summation takes over all the  $(2k+1)$-elements permutation of $\{1,2,\cdots, n-1\}$.
We refer the readers to \cite{GWW2} for more details.
\end{proof}

With all above preparing work, we are ready to prove a special case of Theorem \ref{thmLk} first.

\begin{theo}\label{Lkcons.}
Let $1\leq k\leq \frac{n-1}{2}$ be an integer and $\Sigma^{n-1}$ be a closed horospherical convex hypersurface in the hyperbolic space $\H^n$. If $L_k$ with the induced metric on $\S$ is constant,  then $\Sigma$ is a centered geodesic hypersphere.
\end{theo}
\begin{proof}
Since $L_1=R=(n-1)(n-2)(H_2-1),$
it suffices to discuss the remaining case $k\geq 2.$
Observe that (\ref{Minkowski_Identity}) implies
\beq\label{eq1_Lk}
\int_{\S}V\widetilde L_k d\mu=\int_{\S}\<X,\nu\>\widetilde N_k d\mu.
\eeq
The definition of $\widetilde L_k,\,\widetilde N_k$ gives $$\widetilde L_0=1,\widetilde N_0=H_1.$$
Thus using (\ref{Minkowski_Identity}) again, we have
\beq\label{eq2_Lk}
\int_{\Sigma} V \widetilde L_{0}d\mu= \int_{\Sigma}\<X,\nu\> \widetilde N_{0} d\mu.
\eeq
By (\ref{Lk1}), we know that $\widetilde L_k$ is also constant.
Combining (\ref{eq1_Lk}) and (\ref{eq2_Lk}) together, we have
$$\int_{\S}\<X,\nu\>(\widetilde N_k-\widetilde N_0\widetilde L_k)=0.$$
On the other hand, (\ref{key inequ.}) yields
$$\widetilde N_k-\widetilde N_0\widetilde L_k\leq 0.$$
This forces $$\widetilde N_k-\widetilde N_0\widetilde L_k= 0.$$
everywhere in $\Sigma.$ By Proposition \ref{key_lemm}, there are two cases that equality holds. However, we assert that the second case will not happen.
In fact, in case $(ii)$ we have from (\ref{eq001}) below that $$\widetilde L_k\equiv 0,\;\;\forall \,\;k\geq 2.$$ However, in $\mathbb{H}^n$, there exists a horo-elliptic point, where all principal curvatures are strictly larger than $1$ (this follows from the fact $\l'(r)/\l(r)>1$). Hence it follows again from (\ref{eq001}) below that at this point $\widetilde L_k>0.$
We get a contradiction. Therefore we conclude that $\Sigma$ is a geodesic sphere.
\end{proof}
\vspace{2mm}

To prove the rigidity result regarding the general linear combination of $L_k$, Proposition \ref{key_lemm} is not enough. We need to develop the following more general Newton-Maclaurin type inequalities which may have independent interest.
\begin{prop}\label{key_lemm2}
For  any $\kappa$ satisfying
\begin{equation*}
\kappa\in \{\kappa=(\kappa_1,\kappa_2,\cdots,\kappa_{n-1})\in\R^{n-1}\,|\, \kappa_i\ge 1\},
\end{equation*}
we have
\begin{equation}\label{aim}
\widetilde N_{k-1}\widetilde L_k\geq \widetilde N_k \widetilde L_{k-1}.
\end{equation}
Equality holds if and only if
one of the following two cases holds
$$
\hbox{either } \quad (i)\,  \kappa_i=\kappa_j \; \forall \; i,j,  \quad\hbox{ or }\quad (ii)\,k\geq 2,\;\, \exists \, i\, \hbox{ with }\kappa_i >1\, \& \, \kappa_j=1 \; \forall\; j\neq i.$$
\end{prop}

\begin{proof}
Set $$\kappa_i=1+\hat \kappa_i,$$ then $\hat\kappa_i\geq 0$ for any $i\in\{1,\cdots,n-1\}$.
Define
$$\hat H_i:=H_i(\hat\kappa_1,\hat\kappa_2,\cdots,\hat\kappa_{n-1}).$$
Then
$$H_k=\sum_{i=0}^k \binom{k}{i} \hat H_i.$$
thus
\begin{eqnarray}\label{eq001}
\widetilde L_k=\sum_{i=0}^{k} 2^i \binom{k}{i} \hat H_{2k-i},\quad
\widetilde N_k=\sum_{i=0}^{k} 2^i \binom{k}{i} \hat H_{2k+1-i}.
\end{eqnarray}
Observing that $\widetilde L_k$ and $\widetilde N_k$ can be splitted into two terms,
\begin{eqnarray*}
\widetilde L_k &=&\sum_{i=0}^{k-1} 2^i \binom{k-1}{i}\hat H_{2k-i}+2\sum_{i=0}^{k-1} 2^i\binom{k-1}{i}\hat H_{2k-1-i},\\
\widetilde N_k& =&\sum_{i=0}^{k-1} 2^i \binom{k-1}{i}\hat H_{2k+1-i}+2\sum_{i=0}^{k-1} 2^i\binom{k-1}{i}\hat H_{2k-i},\\
\end{eqnarray*}
we introduce the notation
\begin{equation*}%\label{notation}
X_{s,t}=:\sum_{i=0}^t  2^i \binom{t}{i} \hat H_{s+t-i}.
\end{equation*}
It is clear that
\begin{eqnarray*}
\widetilde L_k&=&X_{k+1,k-1}+2X_{k,k-1}, \quad\;\;\widetilde L_{k-1}=X_{k-1,k-1},\\
\widetilde N_k&=&X_{k+2,k-1}+2X_{k+1,k-1}, \quad\widetilde N_{k-1}=X_{k,k-1}.\\
\end{eqnarray*}
Hence the desired result (\ref{aim}) is equivalent to
\begin{equation}\label{equiv.form}
\left(X_{k+1,k-1}+2X_{k,k-1}\right)X_{k,k-1}\geq \left(X_{k+2,k-1}+2X_{k+1,k-1}\right)X_{k-1,k-1}.
\end{equation}
We claim that this is true. In fact, we can show the  more general result as  stated in the following lemma:
\begin{lemm}\label{lemm1}
 For any $s\geq 1$ and $t\geq 0,$
\begin{equation}\label{eq.1}
X_{s,t}^2\geq X_{s+1,t}X_{s-1,t}.
\end{equation}
\end{lemm}
\begin{proof}
We use the induction argument for $t$ to prove this lemma.
When $t=0$, (\ref{eq.1}) holds for any $s\ge1$ by the standard Newton-MacLaurin identity (\ref{NM}).
Assume (\ref{eq.1}) holds for $t$, we need to prove that (\ref{eq.1}) holds for $t+1$.
Observe the relation that
\begin{equation}\label{observation}
X_{s,t+1}=X_{s+1,t}+2X_{s,t}.
\end{equation}
Using the assumption that (\ref{eq.1}) holding for any $s\geq 1$ and fixed $t$, we derive
\begin{eqnarray*}
&&X_{s,t+1}^2- X_{s+1,t+1}X_{s-1,t+1}\\
&=&\left(X_{s+1,t}^2-X_{s+2,t}X_{s,t}\right)+2\left(X_{s+1,t}X_{s,t}-X_{s+2,t}X_{s-1,t}\right)+4\left(X_{s,t}^2-X_{s+1,t}X_{s-1,t}\right)\\
&\geq&0.
\end{eqnarray*}
The proof of the lemma is completed.
\end{proof}
Choosing $t=k-1$ in  (\ref{eq.1}), it is easy to see that (\ref{equiv.form}) holds. Hence we complete the proof of Proposition \ref{key_lemm2}.
\end{proof}
\vspace{2mm}

We are now in a position to prove the general case of Theorem \ref{thmLk}.

\vspace{2mm}
\noindent {\it Proof of  Theorem \ref{thmLk}:}
By (\ref{Lk1}) and (\ref{widetildeL}), the assumption is equivalent to
$$\sum_{j=l}^k \widetilde b_j \widetilde L_j=\sum_{i=0}^{l-1}\widetilde a_i\widetilde L_i,$$
where
$$\widetilde a_i= \binom{n-1}{2i}(2i!)\,a_i,\;\widetilde b_j=\binom{n-1}{2j}(2j!)\,b_j.$$
Inductively using (\ref{aim}), we get
\begin{eqnarray}\label{LN1}
\widetilde L_j \widetilde N_i\geq \widetilde N_j \widetilde L_i,\; \mbox{for}\; j>i,
\end{eqnarray}
thus we have
\begin{eqnarray}
\sum_{i=0}^{k-1}  \widetilde a_i \widetilde N_i \sum_{j=l}^k  \widetilde b_j\widetilde L_j \geq \sum_{j=l}^k \widetilde b_j \widetilde N_j  \sum_{i=0}^{k-1} \widetilde a_i\widetilde L_i.
\end{eqnarray}
Hence
\begin{eqnarray}
\sum_{i=0}^{k-1} \widetilde a_i \widetilde N_i \geq \sum_{j=l}^k\widetilde  b_j\widetilde N_j.
\end{eqnarray}
Therefore applying (\ref{eq1_Lk}), we have
\begin{eqnarray}
0=\int_\Sigma V(\sum_{j=l}^k \widetilde b_j \widetilde L_j-\sum_{i=0}^{k-1} \widetilde a_i\widetilde L_i)d\mu=\int_\Sigma \<X,\nu\>(\sum_{j=l}^k \widetilde b_j \widetilde N_j-\sum_{i=0}^{k-1} \widetilde a_i \widetilde N_i)d\mu\leq 0.
\end{eqnarray}
Arguing  as in the proof of Theorem \ref{Lkcons.}, one can exclude the case (ii) in Proposition \ref{key_lemm2}. Hence we conclude $\S$ is a geodesic sphere.
\qed

\

A suitable adaption of the  above argument allows us to demonstrate the same result in the hemisphere case, Theorem \ref{thmLkNk}.

\vspace{2mm}
\noindent{\it Proof of Theorem \ref{thmLkNk}:} According to the proof of Theorem \ref{thmLk}, it suffices to establish the corresponding inequality of (\ref{aim}) in $\mathbb{S}^n_+$ under the assumption of $2k$-convexity. The proof basically follows from the one of Proposition \ref{key_lemm2} except some modifications, so we briefly sketch it here.
First, using the simple fact $\binom{k}{i}=\binom{k-1}{i}+C_{k-1}^{i-1}$, in view of \eqref{Lk1}, we can split $\widetilde L_k$ and $\widetilde N_k$ into two terms
\begin{eqnarray*}
\widetilde L_k = \sum_{i=0}^{k-1} \binom{k-1}{i} H_{2k-2i}+\sum_{i=0}^{k-1} \binom{k-1}{i} H_{2k-2-2i},\\
\widetilde N_k = \sum_{i=0}^{k-1} \binom{k-1}{i} H_{2k+1-2i}+\sum_{i=0}^{k-1} \binom{k-1}{i} H_{2k-1-2i}.
\end{eqnarray*}
Next we introduce the notation
\begin{equation*}%\label{notation}
X_{s,t}=:\sum_{i=0}^t \binom{t}{i} H_{s+2t-2i}.
\end{equation*}
It is clear that
\begin{eqnarray*}
\widetilde L_k&=&X_{2,k-1}+X_{0,k-1}, \quad\;\;\widetilde L_{k-1}=X_{0,k-1},\\
\widetilde N_k&=&X_{3,k-1}+X_{1,k-1}, \quad\widetilde N_{k-1}=X_{1,k-1}.\\
\end{eqnarray*}
By a similar induction argument as in the proof of Lemma \ref{lemm1}, one can show that
\beq\label{eq111}
\mbox{For any } s\geq 1\;\mbox{and}\; t\geq 0, \quad X_{s,t}X_{s+1,t}\geq X_{s-1,t}X_{s+2,t}.
\eeq
Finally choosing $s=1,t=k-1$ in (\ref{eq111}), we obtain
\begin{equation}\label{aim00}
\widetilde N_{k-1}\widetilde L_k\geq \widetilde N_k \widetilde L_{k-1}.
\end{equation}
We complete the proof.
\qed

\vspace{2mm}
In a similar way, one can also prove the rigidity result for the curvature functions $N_k$. We only state the result here and leave the proof to readers.
\begin{theo}\label{thmNk}Let $1\leq l<k\leq \frac{n-2}{2}$ be two integers and $\Sigma$ be a closed $(2k+1)$-convex $($horospherical convex resp.$)$ hypersurface embedded in the hyperbolic space $\mathbb{S}_+^n$ $($$\H^n$, resp.$)$. If  there are nonnegative constants $\{a_i\}_{i=0}^{l-1}$ and $\{b_j\}_{j=l}^k$,  at least one of them not vanishing, such that
$$\sum_{i=0}^{l-1}a_iN_i=\sum_{j=l}^k b_j N_j,$$
 then $\Sigma$ is a centered geodesic hypersphere. In particular, if $N_k$  is constant,  then $\Sigma$ is a centered geodesic hypersphere.
\end{theo}

We end this paper with a remark.
\begin{RK} By virtue of our main results, Theorems \ref{mainthm0}, \ref{mainthm00}, \ref{thmLk}, \ref{thmLkNk} and \ref{thmNk}, we tend to believe that rigidity holds for hypersurfaces with constant linear combinations of $H_k$, i.e., $$\sum_{i=1}^{n-1} a_k H_k=const.$$ for any $a_k\in \mathbb{R},$ not necessary nonnegative. In fact, Theorem \ref{thmLk}, \ref{thmLkNk} and \ref{thmNk} include a large class of such rigidity results for  linear combinations of $H_k$ with pure even (or odd) indices. However, our method seems not enough to prove the most general version of linear combinations. \end{RK}

\

\noindent{\bf Acknowledgment.}  Both authors would like to thank Professors Guofang Wang, Yuxin Ge and Dr. Wei Wang for helpful discussions.

\

\end{document}